\documentclass[letterpaper, 10 pt, conference]{ieeeconf}  
\IEEEoverridecommandlockouts                              
\let\labelindent\relax

\usepackage[utf8]{inputenc}
\usepackage{graphics} 
\usepackage{amsmath} 
\usepackage{amssymb}  
\usepackage{amsfonts}
\usepackage{comment}
\usepackage{enumitem}
\usepackage{cases,xspace,enumitem}
\usepackage{stmaryrd}

\usepackage{version}

\includeversion{arxiv}
\excludeversion{lcss}

\usepackage{tikz}
\usetikzlibrary{positioning}

\usepackage[prependcaption,colorinlistoftodos]{todonotes}

\definecolor{gnred}{RGB}{255,91,89}

\definecolor{gnred1}{RGB}{71,0,0} 
\definecolor{gnred2}{RGB}{117,0,0} 
\definecolor{gnred3}{RGB}{164,0,0} 
\definecolor{gnred4}{RGB}{211,0,0} 
\definecolor{gnred5}{RGB}{255,0,0} 
\definecolor{gnred6}{RGB}{255,42,34} 
\definecolor{gnred7}{RGB}{255,91,89} 

\definecolor{gnblue1}{RGB}{0,36,71}   
\definecolor{gnblue2}{RGB}{0,60,118}  
\definecolor{gnblue3}{RGB}{0,85,164}
\definecolor{gnblue4}{RGB}{0,108,212}
\definecolor{gnblue4}{RGB}{0,108,212}
\definecolor{gnblue5}{RGB}{0,133,255}  
\definecolor{gnblue6}{RGB}{35,156,255} 
\definecolor{gnblue7}{RGB}{88,177,255} 

\definecolor{gnbrown1}{RGB}{71,27,0}  
\definecolor{gnbrown2}{RGB}{117,45,0} 
\definecolor{gnbrown3}{RGB}{164,62,0} 
\definecolor{gnbrown4}{RGB}{211,80,0} 
\definecolor{gnbrown5}{RGB}{255,97,0} 
\definecolor{gnbrown6}{RGB}{255,127,26} 
\definecolor{gnbrown7}{RGB}{255,155,86} 

\renewcommand{\theenumi}{(\roman{enumi})}

\newcommand\Item[1][]{%
  \ifx\relax#1\relax  \item \else \item[#1] \fi
  \abovedisplayskip=0pt\abovedisplayshortskip=0pt~\vspace*{-\baselineskip}}

\usepackage{hyperref}
\hypersetup{
    colorlinks=true,
    linkcolor=blue,
    filecolor=magenta,      
    urlcolor=cyan,
    pdftitle={Overleaf Example},
    pdfpagemode=FullScreen,
    }

\usepackage{amsthm}

\newtheoremstyle{ieeeconf}
  {0pt}   
  {0pt}   
  {\normalfont}  
  {\parindent}       
  {\itshape} 
  {:}         
  { } 
  {\thmname{#1} \thmnumber{#2}\thmnote{ (#3)}} 
\makeatletter
\renewenvironment{proof}[1][\proofname]{\par
  \pushQED{\qed}%
  \normalfont \topsep\z@
  \trivlist
  \item[\hskip2em
        \itshape
    #1\@addpunct{:}]\ignorespaces
}{%
  \popQED\endtrivlist\@endpefalse
}
\makeatletter

\theoremstyle{ieeeconf}

\newtheorem{theorem}{Theorem}
\newtheorem{lemma}[theorem]{Lemma}

\newtheorem{assumption}[theorem]{Assumption}

\newtheorem{definition}[theorem]{Definition}

\newtheorem{remark} [theorem]{Remark}

\theoremstyle{myplainroman}

\newcommand{\norm}[1]{\|#1\|}

\newcommand{\diag}[1]{[#1]}

\newcommand{\subscr}[2]{#1_{\textup{#2}}}

\newcommand{\setdef}[2]{\{#1 \, | \, #2\}}

\newcommand{\map}[3]{#1 \colon #2 \rightarrow #3}

\newcommand{\snorm}[1]{{\left\vert\kern-0.25ex\left\vert\kern-0.25ex\left\vert #1 
		\right\vert\kern-0.25ex\right\vert\kern-0.25ex\right\vert}}

\newcommand{\bfx}{\mathbf{x}}

\newcommand{\real}{\mathbb{R}}

\newcommand{\image}{\operatorname{img}}

\newcommand{\mcA}{\mathcal{A}}
\newcommand{\mcS}{\mathcal{S}}
\newcommand{\mcK}{\mathcal{K}}

\newcommand{\mcQ}{\mathcal{Q}}

\newcommand{\qmin}{\subscr{q}{min}}
\newcommand{\qmax}{\subscr{q}{max}}

\newcommand{\amin}{\subscr{a}{min}}
\newcommand{\amax}{\subscr{a}{max}}

\newcommand{\lmin}{\subscr{\lambda}{min}}
\newcommand{\pinv}[1]{#1^{\dagger}}

\newcommand{\ones}[1]{\mathbf{1}_{#1}}

\DeclareMathOperator{\BR}{\operatorname{BR}}
\DeclareMathOperator{\FBR}{\operatorname{F_{BR}}}
\DeclareMathOperator{\FPseudoG}{\operatorname{F_{PseudoG}}}
\DeclareMathOperator*{\argmin}{arg\,min}

\title{Contractivity of Distributed Optimization and Nash Seeking
  Dynamics}

\author{Anand Gokhale, Alexander Davydov, Francesco Bullo\thanks{This work was in part supported by AFOSR project FA9550-21-1-0203 and NSF Graduate Research Fellowship under Grant 2139319. The second author thanks Ibrahim Ozaslan and Dr. Mihailo Jovanovi{\' c} for valuable conversations regarding primal-dual dynamics. The third author thanks Dr.~Tamer Ba{\c s}ar for an insightful early conversation.}%
\thanks{Anand Gokhale, Alexander Davydov, and Francesco Bullo are with the Center for Control, Dynamical 
Systems, and Computation, UC Santa Barbara, Santa Barbara, CA 93106 USA. {\tt\small 
anand\_gokhale@ucsb.edu, davydov@ucsb.edu, bullo@ucsb.edu}.}}

\begin{document}
\maketitle
\thispagestyle{empty}
\pagestyle{empty}

\begin{abstract}
  In this letter, we study distributed optimization and Nash
  equilibrium-seeking dynamics from a contraction theoretic
  perspective. Our first result is a novel bound on the logarithmic norm of
  saddle matrices.  Second, for distributed gradient flows based upon
  incidence and Laplacian constraints over arbitrary topologies, we
  establish strong contractivity over an appropriate invariant vector
  subspace.  Third, we give sufficient conditions for strong contractivity
  in pseudogradient and best response games with complete information, show
  the equivalence of these conditions, and consider the special case of
  aggregative games.
\end{abstract}

\section{Introduction}
\textit{Problem Description and Motivation: } The past decade has witnessed
rapid progress in distributed optimization problems, spurred on by the
seminal paper~\cite{AN-AO:09}. This growing body of literature is surveyed
in~\cite{TY-XY-JW-DW:19}. A significant majority of the literature shows
the convergence of distributed optimization algorithms to a fixed
point. Beyond convergence, practical implementations of these algorithms
may require robustness to noise, delays, and unmodeled dynamics, and may
involve time-varying parameters.  No unifying framework currently exists
that achieves these favorable properties in distributed optimization.
Parallel to these developments, game theory and its applications in
multi-agent systems have seen recent significant advancements~\cite{LP:22}
and there continues to be great interest for a ``unified mathematical
framework for learning a Nash equilibrium in games". Recently, game
theoreticians have utilized passivity and its derivatives to study the
convergence of Nash equilibrium (NE) seeking dynamics~\cite{DG-LP:19,
  LP:22}.

In recent years, contraction theory~\cite{WL-JJES:98} has
emerged~\cite{FB:23-CTDS} as a promising framework for a broad range of
dynamical systems, including both distributed optimization dynamics and
NE-seeking dynamics. Specifically, the strong contractivity property is
known to guarantee exponential convergence, robustness in time-varying
optimization problems, periodic entrainment, and robustness to noise~\cite[Section 3.4]{FB:23-CTDS}.  For NE-seeking systems with
time-invariant dynamics, strong contractivity implies the existence and
uniquness of a Nash equilibrium, associated with two Lyapunov functions. In
time-varying systems, the error between the equilibrium trajectory and the
system trajectory is explicitly and uniformly upper
bounded~\cite{AD-VC-AG-GR-FB:23f}.  Finally, contracting systems support
the use of efficient numerical integration
algorithms~\cite{SJ-AD-AVP-FB:21f}.


Strong infinitesimal contraction is a strict property. This letter employs
a relaxed version of contractivity called semicontractivity, whereby the
system is proved to be strongly contracting in a vector subspace of the
system; see the recent theoretical works~\cite{SJ-PCV-FB:19q,
  GDP-KDS-FB-MEV:21m, FB:23-CTDS,WW-JJES:05}.

\textit{Related Work: }
Here, we study continuous-time dynamics for distributed optimization problems. Such dynamics have previously been discussed in~\cite{SK-JC-SM:15, JC-SKN:19}. The distributed optimization problem may be recast as a linearly constrained optimization problem, which in turn may be solved using a primal-dual approach~\cite{KJA-LH-HU:58}. Previous work in establishing the contraction of linearly constrained primal-dual setups include~\cite{GQ-NL:19, AD-VC-AG-GR-FB:23f, HDN-TLV-KT-JJES:18}, but is limited to full row rank constraint matrices. The exponential stability of such setups with rank-deficient constraints is established in~\cite{IKO-MRJ:23}. Here, we extend these results to contractivity in the case of non-full rank constraints. Our analysis requires tools from semicontraction theory~\cite{SJ-PCV-FB:19q, FB:23-CTDS, GDP-KDS-FB-MEV:21m,WW-JJES:05}.

An early influential reference on the application of contraction analysis
to games is~\cite{SL-TB:87}. More recently, several works invoke the use of passivity-based tools~\cite{LP:22, DG-LP:19} in multi-agent game theoretic systems for seeking a Nash equilibrium. With a fresh look at these dynamics, we approach it using ideas from interconnections of contracting systems~\cite{FB:23-CTDS}.

\textit{Contributions:}
Our main contributions in this paper are threefold.  (i) We start with a general result about the lognorms of saddle matrices that appear while studying continuous time primal-dual dynamics with a linear equality constraint. Specifically, we show that these dynamics are semicontracting in the presence of redundant constraints, while also improving upon the best-known rate in~\cite{GQ-NL:19}. The use of redundant constraints allows the extension of contraction theory to distributed optimization. We also consider a non-symmetric (1,1) block in the saddle matrix as this is useful for developing Nash-seeking dynamics. (ii) Next, we consider the applications of these contracting dynamics in different optimization settings. We show that the Jacobian matrices obtained by linearizing the linearly constrained primal-dual flow and distributed optimization both form a saddle matrix and thus have semicontracting dynamics. 
We show the relationship of our contraction rate with network synchronizability and study the effect of modeling network constraints using Laplacian and incidence matrix-based constraints.
(iii) Finally, we explore Nash-seeking dynamics in game theoretic setups. We identify equivalent sufficient conditions for the contraction of best response and pseudogradient dynamics. We also extend these results to aggregative games obtaining explicit conditions for relating the Lipschitz constants and strong convexity parameter of the cost functions to the contractivity of the system.

The paper is organized as follows,  some preliminaries are discussed in Section~\ref{sec:prelims}. Then we study saddle matrices in Section~\ref{sec:saddle}. Next, we discuss the contractivity of primal-dual dynamics and distributed optimization in Section~\ref{sec:optimization} and NE-Seeking dynamics in Section~\ref{sec:game_theory}.

\section{Preliminaries}\label{sec:prelims}
\textit{Notation : }
Given $x\in \real^n$, we define $\diag{x}\in \real^{n\times n}$ to be a diagonal matrix whose diagonal entries equal to $x$. We denote a vector of ones by $\ones{n}\in \real^n$. For a matrix $A\in \real^{n \times n}$, let $\alpha(A)$ denote its spectral abscissa. We denote the inner product $\langle\cdot,\cdot\rangle:\real^n \times \real^n \rightarrow \real$.
\subsection{Semicontraction theory}
\begin{definition}[Seminorm]
	A function $\snorm{\cdot}:\real^n \rightarrow \real $ is a seminorm on $\real^n$ if , for all 
  $v,w\in \real^n$ and $a\in \real$:
	\begin{enumerate}
		\item $\snorm{av} = |a|\snorm{v}$.
		\item $\snorm{v+w} \leq \snorm{v} + \snorm{w}$.
	\end{enumerate}
	The kernel of a seminorm $\snorm{\cdot}$ is given by ${
		\mcK = \{v\in \real^n : \snorm{v} = 0\}}$.
\end{definition}
\begin{definition}[Matrix induced Seminorm]
	Let $\snorm{\cdot}$ be a seminorm on $\real^n$. Then, the induced matrix seminorm associated 
  with $\snorm{\cdot}$ is given by
	\begin{equation*}
		\snorm{A} = \max\setdef{\snorm{Av}}{{\snorm{v} = 1 \text{ and } v\perp \mcK} }.
	\end{equation*}
\end{definition}

\begin{definition}[Matrix log seminorm] 
Let $\snorm{\cdot}$ be a seminorm on $\real^n$ and its corresponding induced seminorm 
on $\real^{n\times n}$. The \emph{log seminorm} associated with $\snorm{\cdot}$ is defined by
\begin{equation*}
	\mu_{\snorm{\cdot}}(A) = \lim_{h\rightarrow 0^+} \frac{\snorm{I + hA} - 1}{h}.
\end{equation*}
\end{definition}
\begin{definition}[Semicontracting and contracting systems]
	A system $\dot x = F(x)$ is said to be 
  \begin{enumerate}
    \item strongly infinitesimally semicontracting with respect to a seminorm 
    $\snorm{\cdot}$ with rate $c \geq 0$ if
    \begin{equation}\label{eq:semicontracting}
      \sup_x \mu_{\snorm{\cdot}}(DF(x)) \leq -c.
    \end{equation}
    \item strongly infinitesimally contracting with respect to a norm 
    $\norm{\cdot}$ with rate $c \geq 0$ if~\eqref{eq:semicontracting} holds
    and $\snorm{\cdot} = \norm{\cdot}.$
  \end{enumerate}
\end{definition}
In this paper, we will consider weighted Euclidean seminorm, i.e., seminorms of the form $\snorm{x} = \norm{Rx}_2$, where $R$ is some matrix. We present an equivalent condition 
for the semicontraction of a system in such norms presented in~\cite{SJ-PCV-FB:19q}.
\begin{lemma}[Demidovich Condition]
Let $\snorm{\cdot}$ be a seminorm associated with a log seminorm $\mu$, 
$R\in \real^{k\times n}, k \leq n$ be a full rank matrix, and $P = R^\top R \in \real^{n \times n}$. 
For each $A\in \real^{n\times n}$, and $c\in \real$ and if $\ker(R)$ is invariant under $A$, then,
\begin{equation*}
	\mu_{2,R}(A) \leq c \iff PA +  A^\top P \preceq 2cP
\end{equation*}

\end{lemma}
We also consider the following result relating semicontracting systems with an invariance property and contracting systems from~\cite{GDP-KDS-FB-MEV:21m}.
\begin{lemma}[Semicontraction with invariance property]\label{lem:invariance_property}
  Consider a system $\dot x = f(x)$. Let $\mcK \subset \real^n$ be such that $\mcK^{\perp}$
is an $f-$invariant subspace. Let $f:\real^n \rightarrow \real^n$ be strongly infinitesimally semicontracting with rate $c > 0$ with respect to a seminorm $\snorm{\cdot}$ on $\real^n$ with kernel $\mcK$.
Then, the system admits the cascade decomposition,
\begin{subequations}
  \begin{align}
    \dot x_{\|} &= f_{\|} (x_{\|} + x_{\perp})\\
    \dot x_{\perp} &= f_{\perp}(x_{\perp}), \label{eq:perp_dynamics}
  \end{align}
\end{subequations}
and the perpendicular dynamics~\eqref{eq:perp_dynamics} are strongly infinitesimally contracting on $\mcK_{\perp}$ with rate $c$, with respect to the norm given by the domain restriction of $\snorm{\cdot}$ to $\mcK$.
\end{lemma}

\subsection{Network Interconnection Theorem}

\begin{theorem}[Network Contraction Theorem{~\cite[Theorem 3.22]{FB:23-CTDS}}]\label{thm:Network_interaction}
  Consider an interconnection of $r$ dynamical systems,
  \begin{equation*}
    \dot x_i = f_i(t,x_i,x_{-i}) \quad \text{for}\; i \in \{1,2,\hdots,r\}
  \end{equation*}
  where $x_i\in \real^{n_i}$ and $x_{-i}\in \real^{n-n_i}$. We let $x_{-i}$ denote the vector with all components $x_j$ of $x$, except $x_i$. 
  We assume
  \begin{enumerate}
    \item at fixed $x_{-i}$ and $t$, each map $x_i \rightarrow f_i(t,x_i,x_{-i})$ 
    is strongly infinitesimally contracting with rate $c_i$ with respect to $\norm{\cdot}_i$.
    \item at fixed $x_i$ and $t$, each map $x_{-i}\rightarrow f_i(t,x_i,x_{-i})$ is Lipschitz with constant $\gamma_{ij}$.
    \item The gain matrix 
    \begin{equation*}
      \Gamma = \begin{bmatrix}
        -c_1 & \hdots & \gamma_{1r}\\
         \vdots & \ddots & \vdots \\
         \gamma_{r1} & \hdots & -c_r
      \end{bmatrix}
    \end{equation*}
    is Hurwitz.
  \end{enumerate}
  Then, for every $\epsilon \in {]0,\alpha(\Gamma)[}$, there exists a norm such that the interconnected system is strongly infinitesimally 
  contracting with a rate $|\alpha(\Gamma) + \epsilon|$.
\end{theorem}
\section{Logarithmic norm of saddle matrices with redundant constraints, and asymmetry}\label{sec:saddle}
We begin with a general result about Hurwitz saddle matrices. We improve upon the bound for the contraction rate presented 
in~\cite{AD-VC-AG-GR-FB:23f} using a similar proof technique, with extensions to allow for redundant constraints by modifying the weight matrix.
\begin{theorem}[Semicontractivity of Saddle Matrices]\label{thm:quarter_primal_dual}
Given $Q\in\real^{n \times n}$, $A\in\real^{m\times{n}}$, and a parameter
$\tau>0$, consider the \emph{saddle matrix}
\begin{equation*}
  \mcS = \begin{bmatrix}
    -Q & -A^\top \\
    \tau^{-1} A & 0
  \end{bmatrix} \in \real^{(m+n)\times (m+n)}.
\end{equation*}
\begin{enumerate}
\item Let $\lmin(Q+Q^\top)/2 = \qmin>0$ and $Q^\top Q \preceq \qmax (Q+Q^\top)/2$, that is,
  $\qmax=\subscr{\sigma}{max}^2(Q)/\qmin$.
\item Let $A A^\top \succeq \amin \Pi_A$, with $\amin>0$, and $A A^\top \preceq
  \amax I_m$, where the matrix $\Pi_A\in\real^{m\times{m}}$ is the
  orthogonal projection onto the image of $A$, that is, in the linear map
  interpretation, $\map{\Pi_A}{\real^m}{\image(A)}$. Note $\amin\leq\amax$.
\end{enumerate}
The following \emph{semicontractivity LMI} holds:
\begin{equation}
  \label{eq:saddle-semicontract}
    \mcS^\top P + P\mcS\preceq -2 c P ,
  \end{equation}
  where $P\in\real^{(n+m)\times(n+m)}$ is defined by
  \begin{equation}\label{def:P}
    P = \begin{bmatrix}
      I_n & \alpha A^\top \\ \alpha A & 
      \tau \Pi_A \end{bmatrix} \succeq 0
  \end{equation}
  such that $\mcS \ker(P)\subseteq \ker(P)$, with
  \begin{equation*}
    \alpha=\frac{1}{2}\min\Big\{\frac{1}{\qmax},\tau\frac{\qmin}{\amax}\Big\}, \quad
    \text{and} \quad
    c=\frac{1}{2}\tau^{-1}\alpha\amin.
  \end{equation*}
  \end{theorem}
  \begin{proof}
    The proof is deferred to the Appendix~\ref{sec:appendix}.
  \end{proof}
  \begin{arxiv}
    Next, we discuss methods to sharpen the rate obtained. First, we identify an $\alpha$ that achieves a better rate, in exchange for a more complicated expression. Notably, this rate improves upon the previous rate by a factor of $4/3$ when $\tau$ is arbitrarily small, i.e. the dual dynamics are much faster than the primal dynamics. 
    \begin{theorem}[Semicontractivity of Saddle Matrices for small $\tau$]~\label{thm:third_rate_primal_dual}
      Given $Q\in\real^{n \times n}$, $A\in\real^{m\times{n}}$, and a parameter
      $\tau>0$, consider the \emph{saddle matrix}
      \begin{equation*}
        \mcS = \begin{bmatrix}
          -Q & -A^\top \\
          \tau^{-1} A & 0
        \end{bmatrix} \in \real^{(m+n)\times (m+n)}.
      \end{equation*}
      \begin{enumerate}
      \item Let ${\lmin(Q+Q^\top)/2 = \qmin>0}$ and
        ${Q^\top Q \preceq \qmax (Q+Q^\top)/2}$, that is,
        ${\qmax=\subscr{\sigma}{max}^2(Q)/\qmin}$. Note $\qmin\leq\qmax$.
      \item Let $A A^\top \succeq \amin \Pi_A$, with $\amin>0$, and ${A A^\top \preceq
        \amax I_m}$, where the matrix $\Pi_A\in\real^{m\times{m}}$ is the
        orthogonal projection onto the image of $A$, that is, in the linear map
        interpretation, $\map{\Pi_A}{\real^m}{\image(A)}$. Note ${\amin\leq\amax}$.
      \end{enumerate}
      
      Then, the following \emph{semicontractivity LMI} holds:
      \begin{equation*}
          \mcS^\top P + P\mcS\preceq -2 c P ,
        \end{equation*}
        where $P\in\real^{(n+m)\times(n+m)}$ is defined by as per equation~\eqref{def:P}.
        The rate $c=\tfrac{1}{2}\tau^{-1}\alpha\amin$, and 
        \begin{align*}
          \alpha = \min\left\{{\frac{\epsilon}{\qmax}, \frac{2-\epsilon}{3}\tau \frac{\qmin}{\amax}}\right\}
        \end{align*}
        for $0 < \epsilon < 2$.
    \end{theorem}
    \begin{proof}
    The proof follows similar to that of Theorem~\ref{thm:quarter_primal_dual}, and the differences in proofs are discussed in Appendix~\ref{sec:appendix_proof_third} 
    \end{proof}
    
    \begin{remark}
    In the case where ${\tau \rightarrow 0 }$, we get an improved rate of convergence $c$ that approaches $\frac{1}{3} \frac{\amin}{\amax}\qmin$,  for an appropriately chosen $\epsilon$.
    \end{remark}
    
    Finally, we identify a sharper bound that depends on a quadratic constraint, complicating the expression of the bound even further.
    \begin{theorem}[Semicontractivity of Saddle Matrices with a sharper rate]\label{thm:half_rate_primal_dual}
    Given $Q\in\real^{n \times n}$, $A\in\real^{m\times{n}}$, and a parameter
    $\tau>0$, consider the \emph{saddle matrix}
    \begin{equation*}
      \mcS = \begin{bmatrix}
        -Q & -A^\top \\
        \tau^{-1} A & 0
      \end{bmatrix} \in \real^{(m+n)\times (m+n)}.
    \end{equation*}
    \begin{enumerate}
      \item Let ${\lmin(Q+Q^\top)/2 = \qmin>0}$ and
        ${Q^\top Q \preceq \qmax (Q+Q^\top)/2}$, that is,
        ${\qmax=\subscr{\sigma}{max}^2(Q)/\qmin}$. Note $\qmin\leq\qmax$.
      \item Let $A A^\top \succeq \amin \Pi_A$, with $\amin>0$, and ${A A^\top \preceq
        \amax I_m}$, where the matrix $\Pi_A\in\real^{m\times{m}}$ is the
        orthogonal projection onto the image of $A$, that is, in the linear map
        interpretation, $\map{\Pi_A}{\real^m}{\image(A)}$. Note ${\amin\leq\amax}$.
      \end{enumerate}

    Then, the following \emph{semicontractivity LMI} holds:
    \begin{equation*}
        \mcS^\top P + P\mcS\preceq -2 c P ,
      \end{equation*}
      where $P\in\real^{(n+m)\times(n+m)}$ is defined by as per equation~\eqref{def:P}.
      The rate $c=\tfrac{1}{2}\tau^{-1}\alpha\amin$, and $\alpha$ depends on the roots of the following quadratic equation,
      \begin{align}\label{eq:half_rate_quadratic}
          2 - \alpha (\qmax +3\tau^{-1} \frac{\amax}{\qmin}) + \tau^{-1} \alpha^2 \amin = 0.
      \end{align}
    This quadratic has two roots, $\beta_1 < \beta_2$, and ${\alpha = \min\left\{\beta_1, \frac{\tau \qmin}{\amin} \right\}}$.
     \end{theorem}
    \begin{proof}
      The proof follows similar to that of Theorem~\ref{thm:quarter_primal_dual}, and the differences in proofs are discussed in Appendix~\ref{sec:appendix_proof_half} 
     \end{proof}
    Our rate improves upon the rate presented in~\cite{GQ-NL:19}. In order to work with a simple interpretable expression, our analysis in the rest of the paper uses the result in Theorem~\ref{thm:quarter_primal_dual}. Since our norm allows the use of redundant constraints, we may consider problems with a consensus constraint over a graph, such as distributed optimization problems. Further, nonsymmetry of the $(1,1)$ block in the saddle matrix finds applications in NE-seeking dynamics in games.
    
  \end{arxiv}
\begin{lcss}
  We discuss additional sharper results for the contraction rate, including a faster contraction rate for small $\tau$ in the arxiv version\footnote{See https://arxiv.org/abs/2309.05873 for sharper results.}. Our rate improves upon the rate presented in~\cite{GQ-NL:19} by a factor of 2. Since our norm allows the use of redundant constraints, we may consider problems with a consensus constraint over a graph, such as distributed optimization problems. Further, nonsymmetry of the saddle matrix finds applications in NE-seeking dynamics in games.

\end{lcss}
\section{Contractivity in Optimization Dynamics}\label{sec:optimization}
In this section, we study the implications of
Theorem~\ref{thm:quarter_primal_dual} on constrained optimization problems. Specifically, we establish contractivity for the linearly constrained primal-dual flow and the distributed optimization problem. We prove contractivity on a subspace which includes the primal space by proving semicontractivity on the full space. Apart from guaranteeing convergence, showing contractivity leads to useful properties, such as robustness to noise, delays, and strong guarantees over equilibrium tracking and periodic entrainment to time-varying parameters~\cite[Section 3.4]{FB:23-CTDS}.
\subsection{Standard Primal-Dual Dynamics}
Consider the primal-dual flow for a linearly constrained optimization problem~\cite{KJA-LH-HU:58},
\begin{subequations}\label{eq:primal_dual}
\begin{align}
  \dot{x} &=  - \nabla f(x) -A^\top \lambda \\
  \tau \dot{\lambda} &= Ax-b
\end{align}
\end{subequations}
This system, when linearized results in a saddle matrix, therefore, we may use Theorem~\ref{thm:quarter_primal_dual} in order to show that these dynamics are semicontracting.

\begin{lemma}[Semicontraction of primal-dual dynamics]
 If the function $f(x)$ is strongly convex with parameter $\mu$ and Lipschitz with constant $\ell$, 
 and $\amin$, $\amax$ respectively are the smallest and largest nonzero eigenvalues of $AA^\top$, 
 then the dynamics~\eqref{eq:primal_dual} that solve the primal-dual optimization problem in continuous time are semicontracting in the Euclidean norm weighted by $P$, where $P$ is defined in~\eqref{def:P}, with a rate
 \begin{equation*}
 		c = \frac{1}{4}\min\left\{\frac{\amin}{\tau \ell}, \frac{\amin}{\amax}\mu \right\}. 
 \end{equation*}
\end{lemma}
\begin{proof}
	The Jacobian of the system is given by
	\begin{equation*}
		\mcS = 
		\begin{bmatrix}
		- \nabla^2 f(x) & -A^\top \\
		\tau^{-1}A & 0
		\end{bmatrix}.
	\end{equation*}
	The Jacobian is in the form of a saddle matrix, and ${\mu I \preceq \nabla^2 f(x) \preceq \ell I}$. Using Theorem~\ref{thm:quarter_primal_dual}, we get the required result.
\end{proof}
\subsection{Distributed Optimization Dynamics}
In our distributed optimization setup, a team of agents seeks to minimize a decomposable function of the form $f(x) = \sum_{i=1}^{N} f_i(x)$, where $f_i: \real^n \rightarrow \real$, and the sum of all cost functions is \emph{strongly convex}. In this setup, each agent has a local estimate $x_{[i]} \in \real^n$ of the solution to the problem, and agents exchange information over a network with a \emph{strongly connected graph} with a Laplacian $L$. This setup can be represented as the following optimization problem.
\begin{subequations}
\begin{align}
	 	\min_{x \in \real^{nN}} \quad &\sum\nolimits_{i=1}^{N} f_i (x_{[i]})\nonumber\\
		\text{subject to } \quad &(L \otimes I_n)x = 0 \nonumber
\end{align} 
\end{subequations}
We observe that this is similar to a linearly constrained convex optimization problem. Therefore, we may use a primal-dual flow to solve this system. The corresponding dynamics to solve this problem are of the form,
\begin{subequations}\label{eq:distr_primal_dual_laplacian}
\begin{align}
	\dot x &= - \nabla f(x) - (L\otimes I_n)^\top \lambda \\
	\tau \dot \lambda &= (L\otimes I_n) x 
\end{align}
\end{subequations}
Note that these dynamics allow the graph to be directed. However, to obtain an accurate minimizer for the cost function, the graph must be strongly connected. We also make the following assumption about the cost function.
\begin{assumption}\label{asmp:distr_convex}
 $\sum\nolimits_{i=1}^{N} f_i(x_{[i]})$ is twice differentiable, strongly convex with parameter $\mu$ and is Lipschitz in $x$ with parameter $\ell$.
\end{assumption}
We may now study dynamics~\eqref{eq:distr_primal_dual_laplacian} using Theorem~\ref{thm:quarter_primal_dual}.
 \begin{theorem}[Distributed Optimization with Laplacian Constraints]~\label{thm:distr_opt_laplacian}
	The continuous time dynamics~\eqref{eq:distr_primal_dual_laplacian} are semicontracting in a weighted $L_2$ norm with weight 
	\begin{equation*}\label{eq:def_P_distr_opt}
		P = 	\begin{bmatrix}
				I & \alpha (L\otimes I_n)^\top  \\
				\alpha (L\otimes I_n)  &\tau \Pi_{L\otimes I_n}
			\end{bmatrix}
	\end{equation*} 

	with a rate 
	  \begin{equation*}
    		c = \frac{1}{4}\min\Big\{\frac{\lambda_2^2}{\tau\ell},\frac{\lambda_2^2}{\lambda_n^2}\qmin\Big\}.
 	 \end{equation*}
	  where $\lambda_2$ and $\lambda_n$ are the smallest and largest non-zero eigenvalues of $L$ respectively. 
 \end{theorem}
 \begin{proof}
 The linearized version of this system is
\begin{align*}
	\begin{bmatrix}
		\dot{ \delta  x} \\
		\dot{ \delta \lambda}
	\end{bmatrix}&= 
	\begin{bmatrix}
		-\nabla^2 f(x) &  -(L\otimes I_n)^\top\\
		\tau^{-1} (L\otimes I_n) & 0
	\end{bmatrix}	\begin{bmatrix}\delta x\\\delta \lambda\end{bmatrix}
\end{align*}

 	We know that $\nabla^2 f(x)$ is strongly convex with parameter $\mu$ and Lipschitz with parameter $\ell$. 
  Therefore, $\qmin = \mu,\qmax = \ell$. Further, from the properties of the Kronecker product, 
	\begin{align*}
	(L\otimes I)(L\otimes I_n)^\top &= (L\otimes I_n)(L^\top \otimes I_n) = LL^\top \otimes I_n 
	\end{align*}
	and  $LL^\top \otimes I_n $ has the same eigenvalues as the $LL^\top$, but each with multiplicity $n$ times the original 
  multiplicity. The result follows from an application of Theorem~\ref{thm:quarter_primal_dual}. 
\end{proof}
Alternatively, we may model the graph using the incidence matrix $B\in \real^{N \times M}$. Each column of the incidence matrix $B \in \real^{N \times M}$ is associated with an edge in the original graph, and each row is associated with a row in the original matrix. The problem is now represented as,
\begin{align*}
	 	\min_{x \in \real^{nN}} \quad &\sum\nolimits_{i=1}^{N} f_i (x_{[i]})\\
		\text{subject to } \quad &(B^\top \otimes I_n)x = 0 
\end{align*}
The corresponding flow dynamics for this problem are,
\begin{subequations}\label{eq:distr_primal_dual_incidence}
\begin{align}
	\dot x &= - \nabla f(x) - (B^\top\otimes I_n)^\top \lambda \\
	\tau \dot \lambda &= (B^\top\otimes I_n) x 
\end{align}  
\end{subequations}
 \begin{theorem}[Distributed Optimization with incidence matrix constraints]~\label{thm:distr_opt_incidence}
	The continuous time dynamics~\eqref{eq:distr_primal_dual_incidence} are semicontracting in a weighted $L_2$ norm with weight 
	\begin{equation*}
		P = 	\begin{bmatrix}
				I & \alpha (B^\top\otimes I_n)^\top  \\
				\alpha (B^\top \otimes I_n)  &\tau \Pi_{L\otimes I_n}
			\end{bmatrix}
	\end{equation*} 

	with a rate 
	  \begin{equation*}
    		c = \frac{1}{4}\min\Big\{\frac{\lambda_2}{\tau\ell},\frac{\lambda_2}{\lambda_n}\qmin\Big\}.
 	 \end{equation*}
	  where $\lambda_2$ and $\lambda_n$ are the smallest and largest non-zero eigenvalues of $L$ respectively. 
 \end{theorem}
 \begin{proof}
 The linearized version of this system is
\begin{align*}
	\begin{bmatrix}
		\dot{ \delta x }\\
		\dot{ \delta \lambda}
	\end{bmatrix}&= 
	\begin{bmatrix}
		-\nabla^2 f(x) &  -(B\otimes I_n)\\
		\tau^{-1} (B^\top\otimes I) & 0
	\end{bmatrix}	\begin{bmatrix}\delta x\\\delta \lambda\end{bmatrix}
\end{align*}

 	We know that $\nabla^2 f(x)$ is strongly convex with parameter $\mu$ and Lipschitz with parameter $\ell$. Therefore, $\qmin = \mu,\qmax = \ell$. Further, from the properties of the Kronecker product, 
	\begin{align*}
	(B^\top\otimes I_n)(B \otimes I_n)  = B^\top B \otimes I_n 
	\end{align*}
	Now, we know that $B^\top B$ and $BB^\top$ have the same nonzero eigenvalues. Further, $L = BB^\top$. Therefore, $B^\top B \otimes I_n$
	 has the same nonzero eigenvalues as the $L$, but each with multiplicity $n$ times the original multiplicity. The result follows from an application of Theorem~\ref{thm:quarter_primal_dual}. 
\end{proof}

\begin{remark}
  In both Theorem~\ref{thm:distr_opt_laplacian}
  and~\ref{thm:distr_opt_incidence}, the contraction rate depends on
  $\lambda_2/\lambda_n$. Notably, this parameter appears also in the study
  of oscillator networks and their optimal synchronizability~\cite{GC:22}
  via the master stability function approach~\cite{MB-LMP:02}. A higher
  value for $\lambda_2/\lambda_n$ implies both better synchronizability in
  oscillator networks as well as an improved estimate for the contraction
  rate in distributed optimization.
\end{remark}

\subsection{Numerical Simulations}
The lower bound we achieve for the convergence rate using the incidence matrix is faster than the Laplacian constraints, as $\lambda_2/\lambda_n< 1$. To study this comparison further, we considered a simple distributed optimization problem, where each agent attempts to minimize a local function $q_i(x_i - v_i)^2$, for $x_i,v_i,q_i \in \real, q_i > 0$, and $v_i$ and $q_i$ are sampled from uniform distributions in $[0,10]$, and $\tau$ is chosen sufficiently small.  The underlying network is a symmetric connected Erd\H{o}s-R\'enyi graph with $N = 40$ nodes representing agents, with varying edge probability parameters. We construct the saddle matrix of this graph using both the incidence and Laplacian constraints and compute their dominant eigenvalue. On considering 50 graphs for each probability value, the average dominant eigenvalue and its confidence bound are presented in Figure~\ref{fig:lap_inc_comparison}. The Laplacian-based dynamics are faster than the incidence matrix-based dynamics, which implies that the bounds we obtain for the semi-lognorms are weak for this example system. There appears to be an inherent tradeoff between the generality of a bound and its tightness, as observed here, and it remains an open problem to identify better bounds for specific cases.

\begin{figure}
  \centering
  \includegraphics[width=0.95\linewidth]{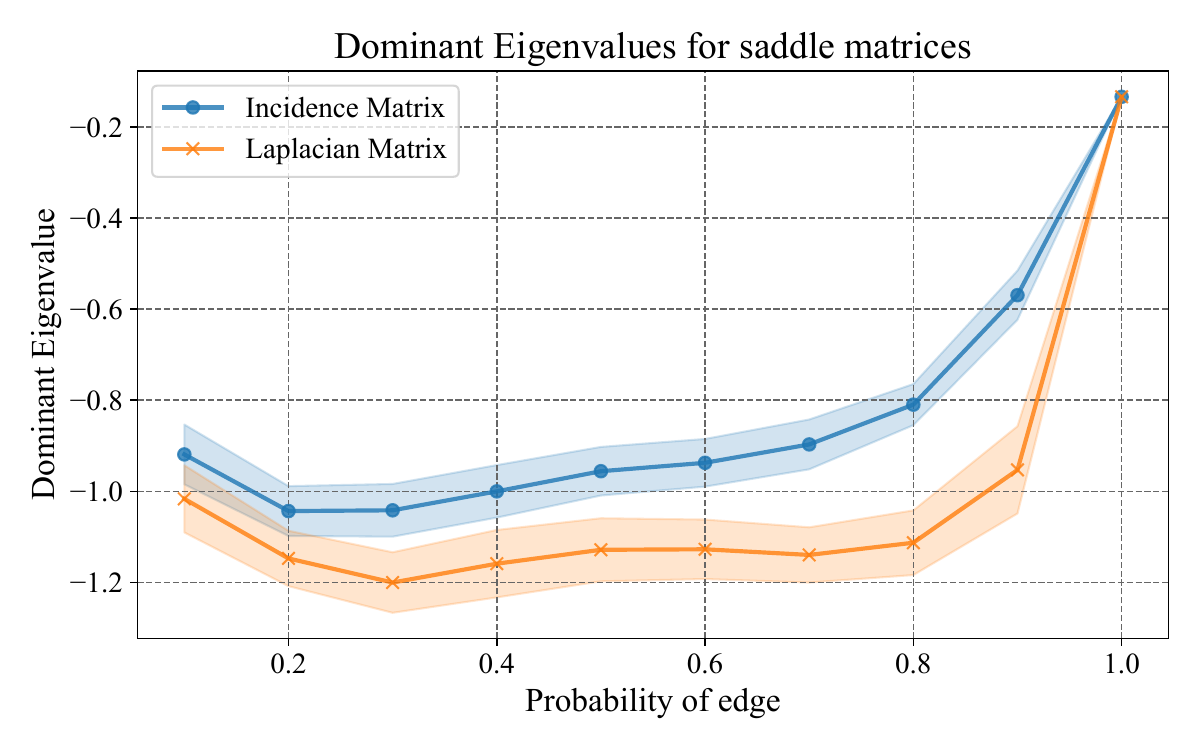}
  \caption{Plot comparing mean dominant eigenvalues of saddle matrices from
    Laplacian and incidence matrix constraints in a distributed
    optimization setup on Erd\H{o}s-R\'enyi graphs with varying edge
    probabilities.}
  \label{fig:lap_inc_comparison}
\end{figure}

\section{Contractivity in Nash seeking dynamics}\label{sec:game_theory}


  


  

In this section, we discuss the contractivity of NE-seeking dynamics that are known to converge to a Nash equilibrium. Contractivity in such systems can be used to prove the robustness of such equilibrium points to delays and noise. 

We consider a multiplayer continuous-time game among ${N}$ players with perfect information exchange. At each instant of time, players pick an action ${x_i\in \real^n}$ 
in order to minimize a cost ${J_i: \real^n \times \real^{n(N-1)} \rightarrow \real}$, which is a function of both the player's actions and the actions of other players. We consider both pseudogradient and best response plays, and use the network contraction theorem to find conditions where they represent strongly infinitesimally contracting dynamics.

First, we consider the evolution of pseudogradient plays. The pseudogradient dynamics are given by 
\begin{equation}\label{eq:pseudogradient_dynamics}
  \dot x_i =\FPseudoG(x) =  - \nabla_{x_i} J_i(x_i,x_{-i})
\end{equation}
where $x_{-i}$ represents the current actions of the other players in the game.
\begin{theorem}[Contraction of Pseudogradient Plays]\label{thm:pseudogradient_contraction}
  Consider a game with $N$ agents, where each agent aims to minimize a cost 
  $J_i: \real^n \times \real^{n(N-1)} \rightarrow \real$, following the dynamics~\eqref{eq:pseudogradient_dynamics}. Then, if $J_i$ is 
  \begin{enumerate}
    \item $\mu_i$ strongly convex with respect to $x_i$ for each $i$, 
    \item Lipschitz with respect to $x_{j}$ for each $j\neq i$ with a constant $\ell_{ij}$, and
    \item the gain matrix for the system, ${
      \subscr{\Gamma}{PseudoG} = \begin{bmatrix}
        -\mu_1 & \hdots & \ell_{1r}\\
         \vdots & \ddots & \vdots \\
         \ell_{r1} & \hdots & -\mu_r
      \end{bmatrix}}$
    is Hurwitz.
  \end{enumerate}
  Then the system is infinitesimally contracting with respect to some norm, with a rate $|\alpha(\Gamma) + \epsilon|$, for $- \alpha(\Gamma) > \epsilon > 0$.
\end{theorem}
\begin{proof}
  The proof for this statement is a direct application of the network contraction theorem.
\end{proof}
Next, we consider the evolution of unconstrained best response plays. We construct a best response cost function,
$\BR_i:x_i \rightarrow\argmin_{x_i} \;{J_i(x_i,x_{-i})}$, which represents the best response of player $i$ with
respect to the actions of other players $x_{-i}$. We represent the best response dynamics,
\begin{equation}~\label{eq:best_response_dynamics}
  \dot x = \FBR(x) = \BR(x) - x 
  \iff \dot x_i = \BR_i(x_{-i}) - x_i 
\end{equation}
\begin{theorem}[Contraction of Best Response Plays]
  Consider a game with $N$ agents, where each agent aims to minimize a cost 
  $J_i: \real^n \times \real^{n(N-1)} \rightarrow \real$, following the dynamics~\eqref{eq:best_response_dynamics}. Then, if $J_i$ is 

  \begin{enumerate}
    \item $\mu_i$ strongly convex with respect to $x_i$ for each $i$, 
    \item Lipschitz with respect to $x_{j}$ for each $j\neq i$ with a constant $\ell_{ij}$, and
    \item the gain matrix for $\FBR$, 
    ${
      \subscr{\Gamma}{BR} = \begin{bmatrix}
        -1 & \hdots & \frac{\ell_{1r}}{\mu_1}\\
         \vdots & \ddots & \vdots \\
         \frac{\ell_{r1}}{\mu_n} & \hdots & -1
      \end{bmatrix}
    }$
    is Hurwitz.
  \end{enumerate}
  Then the system is infinitesimally contracting with respect to some norm, with a rate $|\alpha(\Gamma) + \epsilon|$, for $- \alpha(\Gamma) > \epsilon > 0$.

\end{theorem}
\begin{proof}
  The gain matrix may be computed using the results in~\cite[Lemma 1]{AD-VC-AG-GR-FB:23f}, and~\cite[Equation (3.39)]{FB:23-CTDS}.
  On using the network contraction theorem for this system, we get our required results.
\end{proof}

Next, we show a result about the equivalences of the sufficient conditions for contraction in pseudogradient
and best response plays.
\begin{theorem}[Best Response-Pseudogradient Equivalence]
  If each $J_i$ is
  \begin{enumerate}
    \item $\mu_i$ strongly convex with respect to $x_i$ for each $i$, 
    \item Lipschitz with respect to $x_{j}$ for each $j\neq i$ with a constant $\ell_{ij}$
  \end{enumerate}

  The following statements are equivalent
  \begin{enumerate}
    \item The gain matrix for $\FPseudoG$, ${\subscr{\Gamma}{PseudoG}}$ is Hurwitz.
    \item The gain matrix for $\FBR$, ${\subscr{\Gamma}{BR}}$ is Hurwitz.
    \item The discrete-time $\FBR$ gain matrix ${\subscr{\Gamma}{BR,d} = \begin{bmatrix}
      0 & \hdots & \frac{\ell_{1r}}{\mu_1}\\
       \vdots & \ddots & \vdots \\
       \frac{\ell_{r1}}{\mu_n} & \hdots & 0
    \end{bmatrix}}$ is Schur.
  \end{enumerate}
\end{theorem}
\begin{proof}
  First, we note that $\subscr{\Gamma}{PseudoG} = \diag{\mu}\subscr{\Gamma}{BR}$, where $\mu = [\mu_1, \mu_2, \cdots,  \mu_r]^\top\in \real^n$. Next, we prove $(i)\implies (ii)$. 
  Since $\subscr{\Gamma}{PseudoG}$ is Metzler and Hurwitz, $\subscr{\Gamma}{PseudoG}$ is Lyapunov diagonally stable, 
  i.e. there exists a diagonal matrix $\diag{p}$ such that 
  $\subscr{\Gamma}{PseudoG}^\top \diag{p} + \diag{p}\subscr{\Gamma}{PseudoG} \preceq 0$. Using the relationship 
  between $\subscr{\Gamma}{PseudoG}$ and $\subscr{\Gamma}{BR}$, we note that $\subscr{\Gamma}{BR}$ is Lyapunov diagonally stable
  with weight $\diag{p}\diag{\mu}$. Therefore, $\subscr{\Gamma}{BR}$ is Hurwitz. We may use a similar argument 
  to show that $(ii)\implies (i)$.

  Next, to show that $(ii)\iff(iii)$, we note that $\subscr{\Gamma}{BR} = -I + \subscr{\Gamma}{BR,d}$, 
  and $\subscr{\Gamma}{BR,d}$ is a non-negative matrix.
  We know that $\subscr{\Gamma}{BR}$ is Hurwitz iff $\alpha(\subscr{\Gamma}{BR}) < 0$. Further, since $\subscr{\Gamma}{BR,d}$ is non-negative,
  due to the Perron-Frobenius Theorem, we have $\alpha(\subscr{\Gamma}{BR}) < 0$ iff $\rho(\subscr{\Gamma}{BR,d}) < 1$.
\end{proof}
Next, we discuss the contractivity of aggregative games.
\begin{theorem}[Aggregative games]
  For a pseudogradient play, If  $J_i(x_i,x_{-i})$ is of the form $f_i(x_i,\frac{1}{n}\sum_{j=1}^{n} x_j)$, and for each agent $i$, if
  \begin{enumerate}
    \item $f_i$ is $\mu_i$-strongly convex with respect to its first argument
    \item $f_i$ is Lipschitz in its second argument with constant $\ell_i$, 
    \item $\mu_i > \ell_i$.
  \end{enumerate} 
  Then, the gain matrix of the system is Hurwitz, and the system is contracting.
\end{theorem}

\begin{proof}
  First, we claim that the Lipschitz constant for in $x_{j}$ for $j\neq i$ is $\ell_i/n$. 
  This follows from the fact that the second argument of $f_i$ varies linearly with 
  constant $\frac{1}{n}$ in $x_j$, and our assumption regarding the Lipschitz constant of $f_i$. Further, since $x_i$ appears both in the first and second argument, we may claim that
  \begin{align*}
    &\langle x_i - y_i,\nabla f(x_i,\frac{x_i}{n} + \frac{1}{n}\sum_{j\neq i}x_j) - \nabla f_i(y_i,y_i+ \frac{1}{n}\sum_{j\neq i}x_j)\rangle \nonumber\\
    &\leq -\mu_i\norm{x_i-y_i}^2 + \frac{\ell_i}{n}\norm{x_i - y_i}^2.
  \end{align*}

 Now, we may construct the gain matrix of an aggregative game, $\subscr{\Gamma}{Agg} \in \real^{n\times n}$, 
  \begin{equation*}
    (\subscr{\Gamma}{Agg})_{ij} = \begin{cases}
      - \mu_i + \ell_i / n & i = j,\\
      \ell_i / n & i \neq j
    \end{cases}
  \end{equation*} 
  Since $\mu_i > \ell_i$, $\subscr{\Gamma}{Agg} \ones{n} < 0$, in an elementwise manner. Using properties of Metzler matrices~\cite[Theorem 2.7]{FB:23-CTDS}, $\subscr{\Gamma}{Agg}$ is Hurwitz, and the system is contracting by Theorem~\ref{thm:pseudogradient_contraction}.
\end{proof}

\section{Conclusion}
In this work, we propose the use of contraction theory as a framework to analyze game theoretic and distributed optimization dynamics. We identify a general result for saddle matrices and use it to show contraction in distributed optimization dynamics. Our numerical results suggest that our estimate for the contraction rate is not sharp. Next, we use the network interconnection theorem~\ref{thm:Network_interaction} to show contractivity in some basic game setups. 
We believe that this work opens up new avenues to explore in both distributed optimization and game theory. In distributed optimization, we plan to study other distributed optimization flows~\cite{TY-XY-JW-DW:19}. In game theory, we wish to study setups with partial information and fictitious play~\cite{JSS-GA:05}. We also intend to relax our assumption on strong convexity in future work. 
\appendix
\subsection{Proof of Theorem~\ref{thm:quarter_primal_dual}}\label{sec:appendix}
  We begin by using the Schur Complement to show that $P \succeq 0$. Clearly the $(1,1)$ block 
  is positive definite. Therefore,
\begin{align*}
  P\succeq 0 &\iff \tau \Pi_A - \alpha^2 AA^\top \succ 0 \\
  \quad&\iff \tau - \alpha^2 \amax >0  
  \iff \alpha^2 < \tau /\amax. 
\end{align*}    
 The inequality $\alpha^2 < \tau /\amax$ follows from the stronger inequality  
 $ (2\alpha)^2< \tau/\amax$ with the following argument:
\begin{align*}
  (2\alpha)^2 &\leq
   \min\Big\{\frac{1}{\qmax},\tau\frac{\qmin}{\amax}\Big\}\cdot
   \max\Big\{\frac{1}{\qmax},\tau\frac{\qmin}{\amax}\Big\} \\
  &= \frac{\qmin}{\qmax}\cdot \frac{\tau}{\amax} \leq \frac{\tau}{\amax}.
\end{align*}
Finally, we need to show that ${\mcS\ker(P) \subseteq \ker(P)}$.
Let ${[x^\top \quad y^\top]^\top\in \ker(P)}$. Then ${x = -\alpha A^\top y}$ and 
${y \in \ker(-\alpha AA^\top + \tau \Pi_A)}$. Note that $\Pi_A$ has the same kernel as $AA^\top$. 
Therefore, ${y \in \ker(AA^\top) = \ker(A^\top)}$. Next, to show the invariance of the norm with respect to the dynamics as per Lemma~\ref{lem:invariance_property},
\begin{equation*}
  \mcS\begin{bmatrix} x \\ y\end{bmatrix} = \begin{bmatrix} -A^\top y  \\ 0\end{bmatrix} = 0 \in \ker(P).
\end{equation*}  

Next, we aim to show the LMI~\eqref{eq:saddle-semicontract}. After some bookkeeping, we compute $\mcQ = -  \mcS^\top P - P\mcS - 2 c P$, and using the fact
that $\Pi_A A = A$, 
\begin{align}
  &\mcQ
 = \nonumber \\
 &\begin{bmatrix}
   Q+Q^\top - 2 \tau^{-1} \alpha A^\top A  - 2 c I_n
   &
   \alpha Q^\top A^\top  - 2 c \alpha A^\top
    \\
     \alpha AQ  - 2 c \alpha A
    &  2 \alpha A A^\top   - 2 c \tau \Pi_A
 \end{bmatrix} . \nonumber
\end{align}
The (2,2) block satisfies
the lower bound
\begin{align*}
  2 \alpha A A^\top - 2 c \tau \Pi_A &=
  2\left(\tfrac{1}{2}\alpha A A^\top - c \tau \Pi_A \right) + \alpha A A^\top  \\
  & \succeq 2\big(\tfrac{1}{2}\alpha\amin - c \tau \big)\Pi_A   + \alpha A A^\top \\
  & =  \alpha A A^\top\succ 0.
\end{align*}
Given this lower bound, we can factorize the resulting matrix as follows, 
setting $\mcA = \begin{bmatrix} I_n & 0 \\ 0 & A      \end{bmatrix}$,
\begin{align*}
 &\mcQ 
 \succeq\nonumber\\
  &\mcA
  \begin{bmatrix}
      Q+Q^\top - 2 (\tau^{-1} \alpha A^\top A + c I_n)
      &
      \alpha Q^\top  - 2 c \alpha I_n
      \\
      \alpha Q - 2 c \alpha I_n
      &  \alpha I_n 
  \end{bmatrix}
  \mcA^\top.
\end{align*}
Since $\alpha I_n \succ 0$, it now suffices to show that the
Schur complement of the (2,2) block of $n\times n$ matrix is positive
semidefinite. A sufficient condition for the positive semidefiniteness of the 
Schur complement of the $(2,2)$ block is,
\begin{align}
Q+Q^\top  - \alpha Q^\top Q &\succeq 2 (\tau^{-1} \alpha  A^\top A + c I_n)
\label{ineq:two_quarter_one}\\ 
 \text{and} \quad 2 \alpha c (Q+Q^\top) &\succeq 4 \alpha c^2 I_n.
 \label{ineq:two_quarter_two}
\end{align}
To prove~\eqref{ineq:two_quarter_one}, we upper bound the
right-hand side as follows:
\begin{align*}
 2 (\tau^{-1} \alpha A^\top A +  c I_n) &\preceq
 2 (\tau^{-1} \alpha \amax  +  c ) I_n \\
 &=
 \tau^{-1} \alpha (2\amax + \amin)I_n \\
 & \preceq
 \frac{1}{2} \frac{\qmin}{\amax} (2\amax + \amin)I_n \preceq\frac{3}{2}\qmin I_n.
\end{align*}
where the second equality follows from the definition of $c$ and the last inequality 
follows as $\alpha\leq\tfrac{1}{2}\tau \qmin/\amax$.
Next, since $\alpha\leq\frac{1}{2\qmax}$, we know $- \alpha\qmax \geq
-\frac{1}{2}$. We then lower bound the left hand side of~\eqref{ineq:two_quarter_one} as follows:
\begin{align*}
 Q+Q^\top - \alpha Q^\top Q &\succeq   Q+Q^\top -  \alpha\qmax (Q+Q^\top)/2 \\
 &\succeq (2 - \frac{1}{2}) \frac{1}{2}(Q+Q^\top) \succeq \frac{3}{2}\qmin I_n.
\end{align*}
Finally, it remains to prove~\eqref{ineq:two_quarter_two}
that is, ${2 \alpha c (Q+Q^\top) \succeq 4 \alpha c^2 I_n}$.  This
inequality is equivalent to $Q+Q^\top \succeq 2 c I_n$ and follows from
noting $c\leq \tfrac{1}{2} \tfrac{\amin}{\amax} \qmin < \qmin$.

\begin{arxiv}
  \subsection{Proof of Theorem~\ref{thm:third_rate_primal_dual}}~\label{sec:appendix_proof_third}
  The proof is similar to the previous Theorem~\ref{thm:quarter_primal_dual}, with one key difference. The proof that $\alpha^2 < \tau/\amax$ is modified as follows,
  \begin{align*}
    \alpha^2 &\leq \\
     &\min\Big\{\frac{\epsilon}{\qmax}, \tau\frac{(2-\epsilon)\qmin}{3\amax}\Big\}\cdot
     \max\Big\{\frac{\epsilon}{\qmax}, \tau\frac{(2-\epsilon)\qmin}{3\amax}\Big\} \\
    &= \frac{\qmin}{\qmax}\cdot \frac{\tau}{\amax} \cdot \frac{\epsilon(2 - \epsilon)}{3} \leq \frac{\tau}{3\amax},
  \end{align*}
  which is true since $\epsilon(2 - \epsilon) \leq 1$ for $\epsilon \in (0,2)$. The rest of the proof follows similar to the original proof with some algebraic differences. 
  \subsection{Proof of Theorem~\ref{thm:half_rate_primal_dual}}~\label{sec:appendix_proof_half}
  The proof is very similar to the proof of Theorem~\ref{thm:quarter_primal_dual}  with the following differences. First, we identify a tighter sufficient condition than the
  one presented in~\eqref{ineq:two_quarter_one} and~\eqref{ineq:two_quarter_two}.
  \begin{align}
    \quad (1 + \alpha c)(Q+Q^\top)  - \alpha Q^\top Q &\succeq 2 (\tau^{-1} \alpha  A^\top A + c I_n)\label{ineq:two_half_one}\\
     \text{and} \;  \alpha c (Q+Q^\top) &\succeq 4 \alpha c^2 I_n.
     \label{ineq:two_half_two}
   \end{align}
  The inequality~\eqref{ineq:two_half_two}, $\alpha c (Q+Q^\top) \succeq 4 \alpha c^2 I_n$ is 
  equivalent to $Q+Q^\top \succeq 2 c I_n$ and follows from noting that
  $c \leq \tfrac{1}{2}\qmin$, since we assume $\alpha \leq \tau\qmin / \amin$.
  On simplifying~\eqref{ineq:two_half_one} using similar arguments to the previous proof, we obtain the following condition on $\alpha$,
  \begin{align}
    (2 - \alpha \qmax + \tau^{-1} \alpha^2 \amin)\qmin \geq 3\tau^{-1}\alpha \amax.
  \end{align}
  Next, we note that 
  \begin{align*}
    2 - \alpha (\qmax +3\tau^{-1} \frac{\amax}{\qmin}) + \tau^{-1} \alpha^2 \amin < 0
  \end{align*}
  for $\alpha = \sqrt{\frac{\tau}{\amax}}$. This is because the inequality simplifies to
  \begin{align*}
     \qmax\sqrt{\frac{\tau}{\amax}} +\sqrt{\frac{\amax}{\tau}}\frac{3}{\qmin} > 2 + \frac{\amin}{\amax}
  \end{align*}
  Now, through an AM-GM inequality, the LHS is lower bounded by $2\sqrt{3\tfrac{\qmax}{\qmin}} \geq 2\sqrt{3}$, and the RHS is upper bounded by $3$. Since $2\sqrt{3} > 3$, the quadratic equation is guaranteed to have a root. Since $\alpha \leq \beta_1 < \sqrt{\frac{\tau}{\amax}}$, the quadratic constraint is guaranteed to hold, along with a guarantee of the positive definiteness of $P$.

  \end{arxiv}

\bibliography{alias, Main, FB,New}

\clearpage

\end{document}